\documentclass[12pt,leqno]{article}
\usepackage[T1]{fontenc}
\usepackage{lmodern}
\usepackage{microtype}
\usepackage{amsmath, amsthm, amsfonts, amssymb, color}
\usepackage{mathrsfs}
\usepackage[hidelinks]{hyperref}
\newtheorem{thm}{Theorem}[section]
\newtheorem{cor}[thm]{Corollary}
\newtheorem{lem}[thm]{Lemma}
\newtheorem{prp}[thm]{Proposition}
\newtheorem{exa}[thm]{Example}
\newtheorem{Remark}[thm]{Remark}
\theoremstyle{definition}

\newcommand{\scr}[1]{\mathscr #1}
\definecolor{wco}{rgb}{0.5,0.2,0.3}

\numberwithin{equation}{section}

\newcommand{\ua}{\uparrow}

\title{Monotonicity of Beckner Inequalities
}

\author{Qingbo Lei, Cheng Li, Bo Wu and Jiyang Wu\\
\footnotesize { School  of Mathematical Sciences, Fudan
University, Shanghai 200433, China}\\
{\footnotesize E-mail: 
wubo@fudan.edu.cn
}}

\date{}

\begin{document}

\maketitle

\def\R{\mathbb R} \def\EE{\mathbb E} \def\Z{\mathbb Z} \def\ff{\frac} \def\ss{\sqrt}
\def\H{\mathbb H}
\def\dd{\delta} \def\DD{\Delta} \def\vv{\varepsilon} \def\rr{\rho}
\def\<{\langle} \def\>{\rangle} \def\GG{\Gamma} \def\gg{\gamma}
\def\ll{\lambda} \def\LL{\Lambda} \def\nn{\nabla} \def\pp{\partial}
\def\d{\text{\rm{d}}} \def\loc{\text{\rm{loc}}} \def\bb{\beta} \def\aa{\alpha} \def\D{\scr D}
\def\E{\scr E} \def\si{\sigma} \def\ess{\text{\rm{ess}}}
\def\beg{\begin} \def\beq{\beg}  \def\F{\scr F}
\def\Ric{\text{\rm{Ric}}}
\def\Var{\text{\rm{Var}}}
\def\Osc{\text{\rm{Osc}}}
\def\Ent{\text{\rm{Ent}}}
\def\Hess{\text{\rm{Hess}}}\def\B{\scr B}
\def\e{\text{\rm{e}}} \def\ua{\underline a} \def\OO{\Omega} \def\b{\mathbf b}
\def\oo{\omega}     \def\tt{\tilde} \def\Ric{\text{\rm{Ric}}}
\def\cut{\text{\rm{cut}}} \def\P{\mathbb P} \def\ifn{I_n(f^{\bigotimes n})}
\def\fff{f(x_1)\dots f(x_n)} \def\ifm{I_m(g^{\bigotimes m})} \def\ee{\varepsilon}
\def\C{\scr C}
\def\M{\scr M}\def\ll{\lambda}
\def\X{\scr X}
\def\T{\scr T}
\def\A{\mathbf A}
\def\LL{\scr L}
\def\gap{\mathbf{gap}}
\def\div{\text{\rm div}}
\def\dist{\text{\rm dist}}
\def\cut{\text{\rm cut}}
\def\supp{\text{\rm supp}}
\def\Cov{\text{\rm Cov}}
\def\Dom{\text{\rm Dom}}
\def\Cap{\text{\rm Cap}}
\def\sect{\text{\rm sect}}\def\H{\mathbb H}

\begin{abstract}
In this paper, we prove a monotonicity principle for Beckner inequalities of the form
\[
\frac{\mu(f^2)-\mu(f^p)^{2/p}}{2-p}\leq C\E(f,f),
\qquad f\in\D(\E),\quad 1\le p\le2,
\]
where $\E$ is a conservative symmetric Dirichlet form. If the inequality holds
for every exponent $p$, it holds with the same constant $C$ for all
$q\in[1,p]$. This resolves an open question posed in Chapter~6 of
Wang's monograph~\cite{Wangbook}.

We also study weak Beckner inequalities of the form
\[
\Psi_p(f)\le \beta(r)\E(f,f)+r\,\Osc(f)^2,\qquad f\in\D(\E),\quad r>0,
\]
where $\beta$ is a rate function. We show that the same monotonicity principle
preserves the entire rate function. We further derive consequences for
polynomial rates, semigroup contractivity, and concentration.
\end{abstract}

\vskip 1cm

\section{Introduction}
Beckner inequalities form a one-parameter family that interpolates between the
Poincar\'{e} inequality and the logarithmic Sobolev inequality. Under suitable
Bakry--\'{E}mery curvature conditions, they hold for $1\le p<2$. Introduced in
Beckner's study of sharp inequalities for Gaussian measures, this family has
become a standard tool in the analysis of Markov semigroups; see
\cite{Beckner,BakryEmery,BakryGentilLedoux,Wangbook}. The logarithmic Sobolev
endpoint goes back to Gross~\cite{Gross}. As $p\uparrow2$, the functional
\[\frac{\mu(f^2)-\mu(f^p)^{2/p}}{2-p}\]
converges to one half of the entropy. Thus the endpoint $p=2$ recovers the
logarithmic Sobolev inequality under the normalization adopted below. At
$p=1$, the functional is the variance and yields the Poincar\'{e} inequality.
Related interpolation inequalities and recent developments can be found in
\cite{LatalaOleszkiewicz,BartheCattiauxRoberto,GentilZugmeyer,
AdamczakPolaczykStrzelecki,LamLuRussanov}.

Let $(\Omega,\F,\mu)$ be a probability space. For a non-negative
$f\in L^2(\mu)$ and $1\le p<2$, define
\begin{equation}\label{eq:psi}
  \Psi_p(f)
  :=\frac{\mu(f^2)-\mu(f^p)^{2/p}}{2-p}.
\end{equation}
Then
\[
  \Psi_1(f)=\Var_\mu(f),
\]
while
\[
  \lim_{p\uparrow2}\Psi_p(f)
  =\frac12\Ent_\mu(f^2),
\]
where
\[
  \Ent_\mu(f^2)
  =\mu(f^2\log f^2)-\mu(f^2)\log\mu(f^2).
\]
Throughout the paper, we use the continuous extension
\[
  \Psi_2(f):=\frac12\Ent_\mu(f^2).
\]
Thus, if $\E$ is a non-negative quadratic form, the inequality
\begin{equation}\label{eq:Ip}
  \Psi_p(f)\le C\,\E(f,f)
\end{equation}
interpolates between the Poincar\'{e} inequality and the logarithmic Sobolev
inequality, up to the normalization at $p=2$.

A natural question, raised in connection with the family \eqref{eq:Ip}
and explicitly discussed in \cite{Wangbook}, is whether validity at a larger
exponent implies validity for every smaller one. It is useful here to
distinguish \eqref{eq:Ip} from the modified family
\begin{equation}\label{eq:Iprime}
  (I'_p)\qquad
  p\,\Psi_p(f)
  =
  \frac{p\bigl[\mu(f^2)-\mu(f^p)^{2/p}\bigr]}{2-p}
  \le C\,\E(f,f).
\end{equation}
For the modified functional
\[
  \widetilde\Psi_p(f):=p\,\Psi_p(f),
\]
the map $p\mapsto\widetilde\Psi_p(f)$ is non-decreasing on $[1,2]$;
hence the monotonicity of the corresponding inequalities follows directly from
\cite{Wangbook}. The extra factor $p$ is essential for this argument. For the
unmodified Beckner functional \eqref{eq:psi}, pointwise monotonicity fails,
so a different mechanism is required.

Indeed, let $\gamma$ be the standard Gaussian probability measure on
$\mathbb R$ and let
\[
  f(x)=\Phi(x)
  :=\frac1{\sqrt{2\pi}}\int_{-\infty}^x e^{-t^2/2}\,\d t.
\]
If $X$ is a random variable with the law $\gamma$, then $\Phi(X)$ is uniformly distributed on $(0,1)$, so
\[
  \int_{\mathbb R}f^r\,\d\gamma=\frac1{r+1},
  \qquad r>0.
\]
Consequently,
\[
  \Psi_1(f)=\frac1{12}\approx0.08333,
\]
whereas
\[
  \Psi_{3/2}(f)
  =\frac23-2\left(\frac25\right)^{4/3}
  \approx0.07722.
\]
Hence $\Psi_{3/2}(f)<\Psi_1(f)$. Notice also that
\[
  \int_{\mathbb R}|f'(x)|^2\,\dd\gamma(x)
  =\frac1{2\pi\sqrt3}<\infty,
\]
so the failure already occurs for a smooth function in the natural
Gaussian Sobolev space.

The key observation is that adding a constant does not change the energy of a
conservative Dirichlet form. This suggests replacing \eqref{eq:psi} by the
translation envelope
\begin{equation}\label{eq:Mp}
  M_p(f):=\sup_{c\ge0}\Psi_p(f+c).
\end{equation}
Our main result is the following.

\begin{prp}\label{prop:envelope}
For every bounded non-negative measurable function $f$, the map
\[
  p\longmapsto M_p(f)
\]
is non-decreasing on $[1,2]$, where
$M_2(f)$ is defined by continuous extension.
\end{prp}

As an immediate application we obtain the monotonicity of Beckner
inequalities.

\begin{thm}\label{thm:dirichlet-intro}
Let $(\E,\D(\E))$ be a conservative symmetric Dirichlet form on
$L^2(\mu)$. If \eqref{eq:Ip} holds for some $q\in[1,2]$ with constant
$C$, then it also holds, with the same constant, for every
$1\le p\le q$.
\end{thm}

The proof of Proposition~\ref{prop:envelope} is based on an elementary algebraic
inequality together with the first-order optimality condition for the
translation parameter in \eqref{eq:Mp}.

An analogous result holds for weak Beckner inequalities and is most naturally
stated for an arbitrary rate function. For a bounded non-negative function
$f$, let
\[
  \Osc(f):=\operatorname*{ess\,sup}f-\operatorname*{ess\,inf}f.
\]
The key point is that both the Dirichlet energy and the oscillation are
invariant under addition of constants. Hence the same translation-envelope
argument preserves not only the form of the inequality, but the entire
rate function.

\begin{thm}[Weak Beckner monotonicity with the same rate function]\label{thm:weak-beta-intro}
Let $(\E,\D(\E))$ be a conservative symmetric Dirichlet form on
$L^2(\mu)$, let $p\in[1,2]$, and let
$\beta:(0,\infty)\to[0,\infty)$. Assume that
\begin{equation}\label{eq:weak-beta-intro}
  \Psi_p(f)
  \le
  \beta(r)\,\E(f,f)+r\,\Osc(f)^2,
  \qquad r>0,
\end{equation}
for every bounded non-negative $f\in\D(\E)$. Then, for every $1\le q\le p$,
\begin{equation}\label{eq:weak-beta-intro-q}
  \Psi_q(f)
  \le
  \beta(r)\,\E(f,f)+r\,\Osc(f)^2,
  \qquad r>0,
\end{equation}
with exactly the same rate function $\beta$.
\end{thm}

In Section~\ref{sec:weak} we prove Theorem~\ref{thm:weak-beta-intro} and
specialize it to polynomial rates $\beta(r)=Ar^{-\theta}$. This also
explains the coefficient connecting the usual weak formulation with the
power form and yields monotonicity of the corresponding optimal constants.

\section{Weighted Beckner inequalities}
	
	Let $\mu$ be a probability measure on a measurable space $(\Omega,\F)$, and let
	$(\E,D(\E))$ be a Dirichlet form on $L^2(\mu)$. Only the density of
	$D(\E)$ in $L^2(\mu)$ and the Markov contraction property will be used in
	this section; no symmetry or conservativity assumption is needed.
	
	For the necessity part of the characterization below, we impose the following
	mild non-degeneracy condition:
	\[
	\text{\emph{(A)}}\qquad
	\text{there exist }A_n\subset\Omega\text{ such that }
	0<\mu(A_n)<1
	\quad\text{and}\quad
	\mu(A_n)\longrightarrow1.
	\]
	In particular, \emph{(A)} is satisfied whenever $\mu$ is non-atomic.
	The assumption \emph{(A)} is used only in the necessity direction below.
	
	Let $\phi:[1,2)\to(0,\infty)$ be a positive weight and set
	\[
	I_\phi(p;f)
	:=
	\frac{\phi(p)}{2-p}
	\left[
	\mu(f^2)-\mu(f^p)^{2/p}
	\right]
	=\phi(p)\Psi_p(f),
	\qquad 1\le p<2
	\]
	for non-negative $f\in L^2(\mu)$.
	
	\begin{prp}\label{prop:weighted-characterization}
		Under assumption \emph{(A)}, the map
		\[
		p\longmapsto I_\phi(p;f)
		\]
		is non-decreasing on $[1,2)$ for every non-negative
		$f\in D(\E)\cap L^\infty(\mu)$ if and only if
		\[
		p\longmapsto\frac{\phi(p)}{p}
		\]
		is non-decreasing on $[1,2)$.
	\end{prp}
	
	\begin{proof}
		We first prove the sufficiency. As recalled in the Introduction, the
		modified Beckner functional
		\[
		\widetilde{\Psi}_p(f)
		:=p\Psi_p(f)
		=
		\frac{p}{2-p}
		\left(
		\mu(f^2)-\mu(f^p)^{2/p}
		\right)
		\]
		is non-negative and, for every fixed non-negative $f\in L^2(\mu)$,
		non-decreasing in $p\in[1,2)$. Since
		\[
		I_\phi(p;f)
		=
		\frac{\phi(p)}{p}\,\widetilde{\Psi}_p(f),
		\]
		the conclusion follows immediately whenever
		$p\mapsto\phi(p)/p$ is non-decreasing, because both factors are
		non-negative and non-decreasing. Notice that neither \emph{(A)} nor
		the Dirichlet-form structure is needed for this direction.
		
		We now prove the necessity. Fix $1\le p<q<2$. Let $(A_n)$ be a
		sequence satisfying \emph{(A)}, and write
		\[
		u_n:=\mu(A_n).
		\]
		Thus $0<u_n<1$ and $u_n\to1$.
		
		For each fixed $n$, since $D(\E)$ is dense in $L^2(\mu)$, there exists
		a sequence $g_{n,k}\in D(\E)$ such that
		\[
		g_{n,k}\longrightarrow\mathbf 1_{A_n}
		\qquad\text{in }L^2(\mu)
		\]
		as $k\to\infty$. Let
		\[
		T(t):=(0\vee t)\wedge1
		\]
		and define $f_{n,k}:=T(g_{n,k})$. By the Markov property,
		\[
		f_{n,k}\in D(\E),
		\qquad
		0\le f_{n,k}\le1.
		\]
		Moreover, since $T$ is a contraction and
		$T(\mathbf 1_{A_n})=\mathbf 1_{A_n}$,
		\[
		\|f_{n,k}-\mathbf 1_{A_n}\|_{L^2(\mu)}
		\le
		\|g_{n,k}-\mathbf 1_{A_n}\|_{L^2(\mu)}
		\longrightarrow0.
		\]
		
		For every $r\in[1,2]$, the function $t\mapsto t^r$ is Lipschitz on
		$[0,1]$. Since $\mu$ is a probability measure, the preceding $L^2$
		convergence also implies $L^1$ convergence. Hence
		\[
		\mu(f_{n,k}^r)\longrightarrow\mu(A_n)=u_n
		\qquad\text{as }k\to\infty.
		\]
		In particular,
		\[
		\mu(f_{n,k}^2)\to u_n,
		\qquad
		\mu(f_{n,k}^p)\to u_n,
		\qquad
		\mu(f_{n,k}^q)\to u_n.
		\]
		
		By the assumed monotonicity of $I_\phi(\,\cdot\,;f_{n,k})$,
		\[
		\frac{\phi(p)}{2-p}
		\left(
		\mu(f_{n,k}^2)-\mu(f_{n,k}^p)^{2/p}
		\right)
		\le
		\frac{\phi(q)}{2-q}
		\left(
		\mu(f_{n,k}^2)-\mu(f_{n,k}^q)^{2/q}
		\right).
		\]
		Letting $k\to\infty$, we obtain
		\[
		\frac{\phi(p)}{2-p}
		\left(u_n-u_n^{2/p}\right)
		\le
		\frac{\phi(q)}{2-q}
		\left(u_n-u_n^{2/q}\right).
		\]
		Since $0<u_n<1$,
		\[
		u_n-u_n^{2/r}>0,
		\qquad 1\le r<2,
		\]
		and therefore
		\[
		\frac{\phi(p)(2-q)}{\phi(q)(2-p)}
		\le
		\frac{u_n-u_n^{2/q}}{u_n-u_n^{2/p}}.
		\]
		
		Letting $n\to\infty$ and using $u_n\to1$, l'Hospital's rule gives
		\[
		\begin{aligned}
			\lim_{u\to1^-}
			\frac{u-u^{2/q}}{u-u^{2/p}}
			&=
			\lim_{u\to1^-}
			\frac{1-\frac{2}{q}u^{2/q-1}}
			{1-\frac{2}{p}u^{2/p-1}} =
			\frac{1-\frac{2}{q}}{1-\frac{2}{p}}
			=
			\frac{p(2-q)}{q(2-p)}.
		\end{aligned}
		\]
		Consequently,
		\[
		\frac{\phi(p)(2-q)}{\phi(q)(2-p)}
		\le
		\frac{p(2-q)}{q(2-p)}.
		\]
		Cancelling the positive factor $(2-q)/(2-p)$ yields
		\[
		\frac{\phi(p)}{\phi(q)}\le\frac{p}{q},
		\]
		or equivalently,
		\[
		\frac{\phi(p)}{p}\le\frac{\phi(q)}{q}.
		\]
		Since $1\le p<q<2$ are arbitrary,
		$p\mapsto\phi(p)/p$ is non-decreasing on $[1,2)$.
	\end{proof}
	
	\begin{Remark}\label{rem:weighted-criterion}
		Proposition~\ref{prop:weighted-characterization} gives a criterion
		that does not require differentiability of $\phi$. If $\phi$ is
		differentiable, the condition is equivalently expressed as
		\[
		\left(\frac{\phi(p)}{p}\right)'\ge0
		\quad\Longleftrightarrow\quad
		p\phi'(p)-\phi(p)\ge0
		\quad\Longleftrightarrow\quad
		\phi'(p)\ge\frac{\phi(p)}{p}.
		\]
		
		We have stated the weighted characterization only for $p<2$, since no
		behavior of $\phi$ at the endpoint has been assumed. If $\phi$ admits a
		finite positive extension to $p=2$ and $p\mapsto\phi(p)/p$ remains
		non-decreasing on $[1,2]$, then, with the convention
		\[
		I_\phi(2;f)
		:=\phi(2)\Psi_2(f)
		=\frac{\phi(2)}{2}\Ent_\mu(f^2),
		\]
		the same monotonicity extends to the endpoint $p=2$.
	\end{Remark}
	
	\begin{cor}\label{cor:weighted-beckner}
		Suppose that $p\mapsto\phi(p)/p$ is non-decreasing on $[1,2)$.
		Fix $1\le p\le q<2$. If, for some $C\ge0$,
		\[
		I_\phi(q;f)\le C\E(f,f)
		\]
		holds for every non-negative $f\in D(\E)\cap L^\infty(\mu)$, then
		\[
		I_\phi(p;f)\le C\E(f,f)
		\]
		holds for every such $f$, with the same constant $C$.
	\end{cor}
	
	\begin{proof}
		By the sufficiency argument in Proposition~\ref{prop:weighted-characterization},
		\[
		I_\phi(p;f)\le I_\phi(q;f)\le C\E(f,f).
		\]
	\end{proof}
	
	\begin{Remark}\label{rem:smooth-necessity}
		The necessity argument above is independent of the dimension and of any
		Euclidean structure on $\Omega$. Its only ingredients are the density of
		$D(\E)$ in $L^2(\mu)$ and the Markov contraction property.
		
		Whenever a suitable smooth core with the required approximation property
		is available, the same necessity argument can be implemented with smooth
		test functions. In particular, this is the case for the classical
		Gaussian Dirichlet form on $\mathbb R^d$. Let $\gamma_d$ denote the
		standard Gaussian probability measure and
		\[
		\E(f,f)
		=
		\int_{\mathbb R^d}|\nabla f|^2\,\d\gamma_d,
		\qquad
		D(\E)=H^1(\gamma_d).
		\]
		Choose a sequence of $a_n\in \R$ such that
		\[
		A_n=\{x\in\mathbb R^d:x_1<a_n\},
		\qquad
		\gamma_d(A_n)=u_n\uparrow1,
		\]
		and let $\eta\in C^\infty(\mathbb R)$ satisfy
		\[
		0\le\eta\le1,
		\qquad
		\eta(t)=0\ \text{for }t\le0,
		\qquad
		\eta(t)=1\ \text{for }t\ge1.
		\]
		Then
		\[
		f_{n,k}(x)=\eta\bigl(k(a_n-x_1)\bigr)
		\]
		belongs to $C_b^\infty(\mathbb R^d)\cap H^1(\gamma_d)$ and converges to
		$\mathbf 1_{A_n}$ in $L^2(\gamma_d)$ as $k\to\infty$. Hence the
		necessity conclusion remains valid even if the monotonicity assumption is
		restricted to smooth test functions in the classical Gaussian Dirichlet
		space, in every dimension $d\ge1$.
	\end{Remark}
	
	The characterization also clarifies the transition to the next section. For
	$\phi(p)=p$, the quotient $\phi(p)/p$ is constant, so the modified Beckner
	functional $p\Psi_p(f)$ is pointwise non-decreasing. In contrast, for the
	unmodified Beckner functional $\phi\equiv1$, the quotient $1/p$ is strictly
	decreasing. Thus the unmodified family cannot be handled by pointwise weighted
	monotonicity, which motivates the translation-envelope argument developed in
	the next section.

\section{The translation envelope}

We first record the behavior of $\Psi_p(f+c)$ for large $c$.

\begin{lem}\label{lem:attainment}
Let $f\ge0$ be bounded and let $p\in[1,2]$. Then
$c\mapsto\Psi_p(f+c)$ is continuous on $[0,\infty)$ and
\begin{equation}\label{eq:limit}
  \lim_{c\to\infty}\Psi_p(f+c)=\Var_\mu(f).
\end{equation}
Consequently, after adjoining the point $c=\infty$ and setting
$\Psi_p(f+\infty)=\Var_\mu(f)$, the supremum in \eqref{eq:Mp} is
attained at some $c_p\in[0,\infty]$.
\end{lem}

\begin{proof}
For finite $c$, continuity follows from dominated convergence. We prove
\eqref{eq:limit} for $1\le p<2$; the endpoint $p=2$ follows by the same
second-order expansion.

Write
\[
  f+c=(1+c)(1+zg),\qquad
  g=f-1,\qquad z=(1+c)^{-1},\quad c\geq0.
\]
Since $\Psi_p$ is homogeneous of degree two,
\[
  \Psi_p(f+c)=z^{-2}\Psi_p(1+zg).
\]
As $z\downarrow0$,
\begin{align*}
  \mu((1+zg)^2)
  &=1+2z\mu(g)+z^2\mu(g^2),\\
  \mu((1+zg)^p)
  &=1+pz\mu(g)+\frac{p(p-1)}2z^2\mu(g^2)+O(z^3),
\end{align*}
and therefore
\[
  \mu((1+zg)^p)^{2/p}
  =
  1+2z\mu(g)
  +(p-1)z^2\mu(g^2)
  +(2-p)z^2\mu(g)^2
  +O(z^3).
\]
Substitution in \eqref{eq:psi} gives
\[
  z^{-2}\Psi_p(1+zg)
  =\mu(g^2)-\mu(g)^2+O(z)
  =\Var_\mu(f)+O(z),
\]
which proves \eqref{eq:limit}. The last assertion follows from continuity
on the compactification $[0,\infty]$.
\end{proof}

\begin{proof}[Proof of Proposition~\ref{prop:envelope}]
For $p=1$,
\[
  \Psi_1(f+c)=\Var_\mu(f)
\]
for every $c\ge0$, and therefore
\[
  M_1(f)=\Var_\mu(f)\le M_q(f)
\]
for every $q>1$, by Lemma~\ref{lem:attainment}.

Fix $p\in(1,2)$ and let $c_p\in[0,\infty]$ be a maximizer given by
Lemma~\ref{lem:attainment}. First suppose $c_p<\infty$. Put
\[
  Y=f+c_p,\qquad
  X=\frac{Y}{\mu(Y^p)^{1/p}},
\]
so that $\mu(X^p)=1$. If $c_p>0$, the first-order optimality condition
gives
\[
  \mu(X)=\mu(X^{p-1}).
\]
If $c_p=0$, the right derivative is non-positive, and hence
\begin{equation}\label{eq:opt}
  \mu(X)-\mu(X^{p-1})\le0.
\end{equation}
Thus \eqref{eq:opt} holds in both cases.

Keeping $Y$ fixed and differentiating $\Psi_s(Y)$ with respect to $s$,
we obtain
\begin{equation}\label{eq:derivative}
\left.\frac{\partial}{\partial s}\Psi_s(Y)\right|_{s=p}
=
\mu(Y^p)^{2/p}
\frac{
  \mu(X^2)-1-\frac{2(2-p)}p\mu(X^p\log X)
}{(2-p)^2}.
\end{equation}
Applying Lemma~\ref{lem:algebra} pointwise with $x=X$, integrating
\eqref{eq:fundamental}, and using $\mu(X^p)=1$, we obtain
\[
  \mu(X^2)-1-\frac{2(2-p)}p\mu(X^p\log X)
  \ge
  \eta_p\bigl(\mu(X)-\mu(X^{p-1})\bigr).
\]
Since $\eta_p<0$, the right-hand side is non-negative by
\eqref{eq:opt}. Hence the derivative in \eqref{eq:derivative} is
non-negative.

If $c_p=\infty$, then the active value equals
$\Var_\mu(f)$ by Lemma~\ref{lem:attainment}, and is independent of $p$.
Finally, for $h>0$,
\[
  M_{p+h}(f)-M_p(f)
  \ge
  \Psi_{p+h}(f+c_p)-\Psi_p(f+c_p).
\]
The preceding calculation shows that the lower right Dini derivative of
$p\mapsto M_p(f)$ is non-negative. Since $M_p(f)$ is continuous in
$p$, it follows that $M_p(f)$ is non-decreasing on $[1,2]$.
\end{proof}

\section{Beckner inequalities for conservative Dirichlet forms}

We use the standard terminology of symmetric Dirichlet forms; see
\cite{FukushimaOshimaTakeda}. Let $(\E,\D(\E))$ be a symmetric
Dirichlet form on $L^2(\mu)$. We call
it conservative if $1\in\D(\E)$ and $\E(1,1)=0$. By the
Cauchy--Schwarz inequality for the form,
\[
  \E(f+c,f+c)=\E(f,f),\qquad c\in\mathbb R.
\]

\begin{proof}[Proof of Theorem~\ref{thm:dirichlet-intro}]
Assume first that $f\in\D(\E)$ is bounded and non-negative, and that
\eqref{eq:Ip} holds at the exponent $q$ with constant $C$. For every
$c\ge0$,
\[
  \Psi_q(f+c)
  \le C\,\E(f+c,f+c)
  =C\,\E(f,f).
\]
Taking the supremum over $c$ gives
\[
  M_q(f)\le C\,\E(f,f).
\]
If $1\le p\le q$, Proposition~\ref{prop:envelope} yields
\[
  \Psi_p(f)\le M_p(f)\le M_q(f)\le C\,\E(f,f),
\]
which is exactly the Beckner inequality at exponent $p$.

For a general non-negative $f\in\D(\E)$, apply the bounded result to
$f_n=f\wedge n$. By the Markov property,
$f_n\in\D(\E)$ and $\E(f_n,f_n)\le\E(f,f)$. Moreover
$f_n\to f$ in $L^2(\mu)$ and in $L^p(\mu)$, so
$\Psi_p(f_n)\to\Psi_p(f)$. Passing to the limit proves the claim.
\end{proof}

At the endpoints, the theorem says that the logarithmic Sobolev
inequality implies all intermediate Beckner inequalities and, in
particular, the Poincar\'{e} inequality, with the constants interpreted
according to the normalization
$\Psi_2(f)=\frac12\Ent_\mu(f^2)$.

\section{Weak Beckner inequality}\label{sec:weak}

Weak Poincar\'{e} inequalities with a rate function are a standard tool for
describing subexponential convergence of Markov semigroups
\cite{RocknerWang,Wangbook}. Closely related weak and interpolated functional
inequalities, including their measure--capacity and concentration aspects, are
developed in \cite{BartheCattiauxRoberto,CattiauxGozlanGuillinRoberto}.
We now turn to the corresponding weak Beckner family.
The natural formulation for the present argument uses the oscillation
\[
  \Osc(f):=\operatorname*{ess\,sup} f-\operatorname*{ess\,inf} f.
\]
For a conservative Dirichlet form,
\begin{equation}\label{eq:weak-translation-invariance}
  \E(f+c,f+c)=\E(f,f),
  \qquad
  \Osc(f+c)=\Osc(f),
  \qquad c\in\mathbb R.
\end{equation}
Thus both quantities on the right-hand side of a weak inequality are
translation invariant.

Let $\beta:(0,\infty)\to[0,\infty)$. 
We say that the weak Beckner
inequality $WB_p(\beta)$ holds if
\begin{equation}\label{eq:weak-beta}
  \Psi_p(f)
  \le
  \beta(r)\,\E(f,f)+r\,\Osc(f)^2,
  \qquad r>0,
\end{equation}
for every bounded non-negative $f\in\D(\E)$. No monotonicity assumption on $\beta$ is
needed for the following result.

\begin{proof}[Proof of Theorem~\ref{thm:weak-beta-intro}]
Fix $r>0$. For every $c\ge0$, \eqref{eq:weak-beta-intro} and
\eqref{eq:weak-translation-invariance} give
\[
  \Psi_p(f+c)
  \le
  \beta(r)\,\E(f,f)+r\,\Osc(f)^2.
\]
Taking the supremum over $c\ge0$ yields
\[
  M_p(f)
  \le
  \beta(r)\,\E(f,f)+r\,\Osc(f)^2.
\]
If $1\le q\le p$, Proposition~\ref{prop:envelope} implies
\[
  \Psi_q(f)\le M_q(f)\le M_p(f).
\]
Therefore
\[
  \Psi_q(f)
  \le
  \beta(r)\,\E(f,f)+r\,\Osc(f)^2.
\]
Since $r>0$ is arbitrary, $WB_q(\beta)$ holds with exactly the same
function $\beta$.
\end{proof}

Theorem~\ref{thm:weak-beta-intro} is stronger than a comparison of
particular numerical constants: the entire prescribed weak rate function
$\beta$ is preserved.

We next spell out the polynomial family which connects the customary weak
Poincar\'{e} and weak logarithmic Sobolev inequalities. All inequalities below
are understood for bounded non-negative $f\in\D(\E)$. Consider
\begin{equation}\label{eq:weak-poly}
  \Psi_p(f)
  \le
  \alpha(r)\E(f,f)
  +r\,\Osc(f)^2,
  \qquad r>0.
\end{equation}
The following elementary optimization explains the coefficient appearing
in the equivalent power formulation.

\begin{prp}\label{prop:weak-coefficient}
Fix $p\in[1,2]$. The inequality
\eqref{eq:weak-poly} is equivalent to
\begin{equation}\label{eq:weak-functional}
  \Phi(f)\Psi_p(f)
  \le
  \E(f,f),
\end{equation}
where
\[
\Phi(f)=
\begin{cases}
1, & \Psi_p(f)=0\ \text{or}\ \Osc(f)=0,\\[1mm]
\displaystyle\sup_{t\in(0,1)}
\frac{1-t}{\alpha\!\left(\frac{t\Psi_p(f)}{\Osc(f)^2}\right)},
& \Psi_p(f)\neq0\ \text{and}\ \Osc(f)\neq0.
\end{cases}
\]
\end{prp}

\begin{proof} If $\Psi_p(f)=0$ or $\Osc(f)^2=0$, the
claim is immediate.  Otherwise,
define
\[
  X=\Psi_p(f),\qquad
  Y=\E(f,f),\qquad
  O=\Osc(f)^2.
\]
Assume first that \eqref{eq:weak-poly} holds. Then we have
\[
  Y\ge \frac{X-rO}{\alpha(r)}.
\]
Set
\[
  r=\frac{tX}{O},
  \qquad 0<t<1.
\]
Then
\[
  Y
  \ge
 \frac{X(1-t)}{\alpha\Big(\frac{t X}{O}\Big)},\qquad 0<t<1.
\]
Taking the supremum over $t\in(0,1)$ gives
\[
  \Phi(f)\Psi_p(f)
  \le
  \E(f,f).
\]
This proves one direction. The reverse implication follows from the same
one-variable optimization, equivalently from Young's inequality. Taking
infima over the admissible constants gives
\eqref{eq:weak-functional}.
\end{proof}

Fix $\theta>0$ and
consider
\begin{equation}\label{eq:weak-polynomial-rate}
  \Psi_p(f)
  \le
  A_{p,\theta}\,r^{-\theta}\E(f,f)
  +r\,\Osc(f)^2,
  \qquad r>0.
\end{equation}
The following elementary optimization explains the coefficient appearing
in the equivalent power formulation.

\begin{exa}\label{exam}Fix $p\in[1,2]$.
  \begin{enumerate}
    \item[(1)] Let $\alpha(r)=A_{p,\theta}\,r^{-\theta}$ for some constants $\theta>0$ and $A_{p,\theta}>0$. The inequality
    \eqref{eq:weak-poly} is equivalent to
    \begin{equation}\label{eq:power-form}
      \Psi_p(f)^{1+\theta}
      \le
      C_{p,\theta}\,\E(f,f)\,\Osc(f)^{2\theta}.
    \end{equation}
    For the optimal constants,
    \begin{equation}\label{eq:optimal-constants}
      C_{p,\theta}^*
      =
      A_{p,\theta}^*
      \frac{(1+\theta)^{1+\theta}}{\theta^\theta}.
    \end{equation}
    \item[(2)] Let $\alpha(r)=Ae^{C/r}$ for some constants $A$ and $C>0$. The inequality
    \eqref{eq:weak-poly} is equivalent to
    \begin{equation*}
      \Psi_p(f)\leq \frac{A\exp(\frac{C\Osc(f)^2}{\Psi_p(f)t^*})}{1-t^*}\mathcal{E}(f,f),
    \end{equation*}
    where
    \begin{equation*}
      t^*=\frac{-C\frac{\Osc(f)^2}{\Psi_p(f)}+\sqrt{C^2\frac{\Osc(f)^4}{\Psi_p^2(f)}+4C\frac{\Osc(f)^2}{\Psi_p(f)}}}{2}.
    \end{equation*}
    \item[(3)] Let $\alpha(r)=A(\log\frac{1}{r})^\theta, 0<r\leq1$ for some constants $A$ and $\theta >0$. The inequality \eqref{eq:weak-poly} is equivalent to
    \begin{equation*}
      \Psi_p(f)\leq \frac{A(\log\frac{\Osc(f)^2}{\Psi_p(f)t^*})^{\theta}}{1-t^*}\mathcal{E}(f,f),
    \end{equation*}
    where
    \begin{equation*}
      \frac{1-t^*}{t^*}\theta=\log\left(\frac{\Osc(f)^2}{\Psi_p(f)}\right)-\log t^*.
    \end{equation*}
  \end{enumerate}
\end{exa}

\begin{proof}
  \begin{enumerate}
    \item[(1)] Let $\alpha(r)=A_{p,\theta}\,r^{-\theta}$, then
    \[
    \Phi(f)=
    \begin{cases}
    1, & \Psi_p(f)=0\ \text{or}\ \Osc(f)=0,\\[1mm]
    \displaystyle\sup_{t\in(0,1)}
    \frac{1-t}{A_{p,\theta}
    \left(\frac{t\Psi_p(f)}{\Osc(f)^2}\right)^{-\theta}},
    & \Psi_p(f)\neq0\ \text{and}\ \Osc(f)\neq0.
    \end{cases}
    \]
      The function $t^\theta(1-t)$ is maximized at
      \[
      t=\frac{\theta}{1+\theta},
      \]
      and its maximum is
      \[
      \frac{\theta^\theta}{(1+\theta)^{1+\theta}}.
      \]
      Hence
      \[
      \Psi_p(f)^{1+\theta}
      \le
      A_{p,\theta}
      \frac{(1+\theta)^{1+\theta}}{\theta^\theta}
      \mathcal{E}(f,f) \Osc(f)^{2\theta}.
      \]
    \item[(2)] Let $\alpha(r)=Ae^{C/r}$, then
    $$\Phi(f)=\begin{cases}&1, ~~~~~~~~~~~~~~~~~~\quad\quad\quad\quad\quad\quad\quad\Psi_p(f)=0~\text{or}~\Osc(f)=0,\\
      &\displaystyle\sup_{t\in (0,1)}\frac{1-t}{A\exp\Big(\frac{C\Osc(f)^2}{\Psi_p(f) t}\Big)}, ~~~~~~~~\quad\Psi_p(f)\neq0~\text{or}~\Osc(f)\neq0.\end{cases}$$
      Setting $c'=-C\frac{\Osc(f)^2}{\Psi_p(f)}<0,$ then
      \begin{equation*}
        h(t)=\frac{1-t}{A\exp\Big(\frac{C\Osc(f)^2}{\Psi_p(f) t}\Big)}=\frac{(1-t)}{A}e^{c'/t}.
      \end{equation*}
      Thus
      \begin{equation*}
        h'(t)=-\frac{e^{c'/t}}{At^2}(t^2+c'(1-t)).
      \end{equation*}
      The function $h(t)$ is maximized at
      \begin{equation*}
        t=\frac{c'+\sqrt{c'^2-4c'}}{2}
      \end{equation*}
      and its maximum is
      \begin{equation*}
        \frac{1-t^*}{A\exp(\frac{C\Osc(f)^2}{\Psi_p(f)t^*})}.
      \end{equation*}
      Hence
      \begin{equation*}
        \Psi_p(f)\leq \frac{A\exp(\frac{C\Osc(f)^2}{\Psi_p(f)t^*})}{1-t^*}\mathcal{E}(f,f).
      \end{equation*}
    \item[(3)] Let $\alpha(r)=A(\log\frac{1}{r})^{\theta},$ then
      $$\Phi(f)=\begin{cases}&1, ~~~~~~~~\quad\quad\quad\quad\quad\quad\quad\Psi_p(f)=0~\text{or}~\Osc(f)=0,\\
      &\displaystyle\sup_{t\in (0,1)}\frac{1-t}{A(\log\frac{\Osc(f)^2}{\Psi_p(f)t})^\theta}, ~~\quad\Psi_p(f)\neq0~\text{or}~\Osc(f)\neq0.\end{cases}$$
      Setting $c=\log\frac{\Osc(f)^2}{\Psi_p(f)}>0,$ then
      \begin{equation*}
        h(t):=\frac{1-t}{A(c-\log t)^{\theta}},
      \end{equation*}
      thus
      \begin{equation*}
        h'(t)=\frac{-(c-\log t)+\frac{1-t}{t}\theta}{A(c-\log t)^{\theta+1}}.
      \end{equation*}
      The function $h(t)$ is maximized at
      \begin{equation*}
        \frac{1-t}{t}\theta=c-\log t.
      \end{equation*}
      Hence
      \begin{equation*}
      \Psi_p(f)\leq \frac{A(\log\frac{\Osc(f)^2}{\Psi_p(f)t^*})^{\theta}}{1-t^*}\mathcal{E}(f,f)
    \end{equation*}
  \end{enumerate}
\end{proof}

At $p=1$, \eqref{eq:weak-polynomial-rate} reduces to the weak Poincar\'{e} inequality,
\[
  \Var_\mu(f)
  \le
  A_{1,\theta}\,r^{-\theta}\E(f,f)
  +r\,\Osc(f)^2,
  \qquad r>0,
\]
while at $p=2$ it becomes the weak logarithmic Sobolev inequality
\[
  \frac12\Ent_\mu(f^2)
  \le
  A_{2,\theta}\,r^{-\theta}\E(f,f)
  +r\,\Osc(f)^2,
  \qquad r>0.
\]
Thus \eqref{eq:weak-polynomial-rate} interpolates between the two weak endpoints.

The same-rate theorem immediately gives the monotonicity of both the
polynomial weak rates and the equivalent power constants.

\begin{cor}\label{cor:weak-poly-monotonicity}
Fix $\theta>0$. If \eqref{eq:weak-polynomial-rate} holds at an exponent
$p\in[1,2]$ with a coefficient $A$, then it holds for every
$1\le q\le p$ with the same coefficient $A$. Consequently,
\[
  A_{q,\theta}^*\le A_{p,\theta}^*,
  \qquad
  C_{q,\theta}^*\le C_{p,\theta}^*,
  \qquad
  1\le q\le p\le2.
\]
Equivalently, for each fixed $\theta>0$, both
\[
  p\longmapsto A_{p,\theta}^*
  \quad\text{and}\quad
  p\longmapsto C_{p,\theta}^*
\]
are non-decreasing on $[1,2]$.
\end{cor}

\begin{proof}
Apply Theorem~\ref{thm:weak-beta-intro} with
$\beta(r)=Ar^{-\theta}$. The monotonicity of $A_{p,\theta}^*$ follows
immediately. Since the factor in \eqref{eq:optimal-constants} depends only
on $\theta$, the same monotonicity holds for $C_{p,\theta}^*$.
\end{proof}

In particular,
\[
  A_{1,\theta}^*
  \le
  A_{p,\theta}^*
  \le
  A_{2,\theta}^*,
  \qquad
  C_{1,\theta}^*
  \le
  C_{p,\theta}^*
  \le
  C_{2,\theta}^*.
\]

\section{Contractivity of the semigroup \texorpdfstring{$P_t$}{Pt}}
Let $(P_t)_{t\ge0}$ and $L$ denote, respectively, the Markov semigroup and
generator associated to $(\E,\D(\E))$. We use the following assumptions:
\begin{enumerate}
  \item[(a)] $\mathcal{E}(f^q,f)=\frac{4q}{(1+q)^2}\mathcal{E}(f^{(1+q)/2},f^{(1+q)/2})$ for any $q>0$ and any uniformly positive $f\in L^{\infty}(\mu)\cap\D(L)$.
  \item[(b)] $P_t$ is symmetric.
\end{enumerate}
\begin{thm}
    If (a) holds, then \eqref{eq:weak-poly} is equivalent to
    \begin{equation}\label{eq:semigroup}
        \Psi_p((P_tf)^{1/p})-s\Osc(f^{1/p})^2
        \leq \exp\!\left(-\frac{2t}{\alpha(s)}\right)\left[\Psi_p(f^{1/p})-s\Osc(f^{1/p})^2\right].
    \end{equation}
    In the sequel, we assume that $f$ is not constant so that the left-hand side is well-defined. Conversely, if $(b)$ holds, then \eqref{eq:weak-poly} implies \eqref{eq:semigroup}.
\end{thm}
\begin{proof}
  It suffices to consider uniformly positive
  $f\in L^{\infty}(\mu)\cap\mathcal{D}(L)$. For such a function,
    \begin{equation*}
        \frac{d}{dt}\Psi_p((P_tf)^{1/p})=\frac{2}{p(2-p)}\mu((P_tf)^{2/p-1}LP_tf)=\frac{-2}{p(2-p)}\mathcal{E}((P_tf)^{2/p-1},P_tf)
    \end{equation*}
    for $t\geq0$. By Lemma~5.1.3 of \cite{Wangbook} and assumption (b),
    \begin{equation*}
      \mathcal{E}((P_tf)^{(2-p)/p},P_tf)\geq p(2-p)\mathcal{E}((P_tf)^{1/p},(P_tf)^{1/p}),
    \end{equation*}
    while assumption (a) gives equality. Hence
    \begin{equation*}
      \frac{d}{dt}\Psi_p((P_tf)^{1/p})\leq-2\mathcal{E}((P_tf)^{1/p},(P_tf)^{1/p}),
    \end{equation*}
    Combining this estimate with \eqref{eq:weak-poly}, applied to
    $(P_tf)^{1/p}$, gives
    \begin{equation*}
      \frac{d}{dt}\Psi_p((P_tf)^{1/p})\leq\frac{2}{\alpha(s)}s\Osc((P_tf)^{1/p})^2-\frac{2}{\alpha(s)}\Psi_p((P_tf)^{1/p}).
    \end{equation*}
    By Gronwall's lemma, we get 
    \begin{equation*}
        \Psi_p((P_tf)^{1/p})-s\Osc(f^{1/p})^2\leq e^{\frac{-2t}{\alpha(s)}}\left[\Psi_p(f^{1/p})-s\Osc(f^{1/p})^2\right].
    \end{equation*}
    The proof of the converse direction is essentially the same argument run in reverse.
\end{proof}

We show that the following condition is sufficient for our purpose:
\begin{equation}\label{eq:integrability}
  \int_1^{\infty}\frac{\alpha(r)}{r}dr<\infty.
\end{equation}
\begin{thm}
  Assume that\eqref{eq:integrability} holds.
  Then \eqref{eq:weak-poly} implies
  \begin{equation*}
    \Psi_p((P_tf)^{1/p})\leq\xi(t)\Osc(f^{1/p})^2,\qquad t>0,f\in\mathcal{D}(L),
  \end{equation*}
  where $\xi(t):=\inf_{\epsilon\in(0,1)} \frac{1}{\epsilon}\eta^{-1}(2(1-\epsilon)t)$ and  
  \begin{equation*}
    \eta(t):=\int_t^{\infty}\frac{\alpha(r)}{r}dr,\qquad t>0.
  \end{equation*}
\end{thm}
\begin{proof}
  Fix $f\in\mathcal{D}(L)$, and set $O=\Osc(f^{1/p})^{-2}$ and
  $h(t):=\Psi_p((P_tf)^{1/p})$. By \eqref{eq:weak-poly} and the bound
  $\Osc((P_tf)^{1/p})\leq\Osc(f^{1/p})$, we have
  \begin{equation*}
    h'(t)\leq-2\mathcal{E}((P_tf)^{1/p},(P_tf)^{1/p})\leq-\frac{2}{\alpha(r)}h(t)+\frac{2}{\alpha(r)}r\Osc((P_tf)^{1/p})^2.
  \end{equation*}
  Taking $r=\epsilon h(t)O$, we arrive at
  \begin{equation*}
    h'(t)\leq -\frac{2(1-\epsilon)h(t)}{\alpha(\epsilon h(t)O)}.
  \end{equation*}
  Thus,
  \begin{equation*}
    -2(1-\epsilon)t\geq\int_0^t\frac{h'(s)\alpha(\epsilon h(s)O)}{h(s)}ds=\int_{\epsilon h(0)O}^{\epsilon h(t)O}\frac{\alpha(r)}{r} dr\geq-\eta(\epsilon h(t)O).
  \end{equation*}
  This implies that $h(t)\leq \frac{1}{\epsilon}\eta^{-1}(2(1-\epsilon)t)\Osc(f^{1/p})^2$.
\end{proof}

We first consider $\rho\in\mathcal{D}(L_{\mathcal{E}})$ such that
\begin{equation}\label{eq:condition}
    \mathcal{E}((\rho-t)^{+}\wedge s,(\rho-t)^{+}\wedge s)\leq L_{\mathcal{E}}(\rho)^2\mu(t\leq\rho<s+t),\qquad s,t\geq 0,
\end{equation}
which holds in particular for symmetric diffusions and $\rho\in\mathcal{D}(L_{\mathcal{E}})$ with $\mu(\rho=t)=0$ for all $t\in \mathbb{R}$.
\begin{thm}\label{thm:concentration}
    Assume that $(\mathcal{E},\mathcal{D}(\mathcal{E}))$ is a conservative
    Dirichlet form on $L^2(\mu)$, where $\mu$ is a probability measure. Let
    $\rho\in\mathcal{D}(L_{\mathcal{E}})$ be nonnegative, satisfy
    $L_{\mathcal{E}}(\rho)\leq1$, and obey \eqref{eq:condition}. Let $t_0>0$
    be such that $a:=\mu(\rho<t_0)\in(0,1)$. Assume that
    \eqref{eq:weak-poly} holds for some
    decreasing positive function $\alpha$. For fixed $\epsilon\in (0,1)$ and any $m>0,$ define
    \begin{equation*}
      \phi_m(t):=\int_t^{1-a}\frac{p_0+m^{-2}\alpha(\epsilon p_0s)}{(1-\epsilon)p_0 s}ds,\qquad t\in(0,1-a),
    \end{equation*}
    where $p_0=\frac{1}{2-p}\bigl(1-(1-a)^{(2-p)/p}\bigr)$.
    Then
    \begin{equation*}
      \mu(\rho\geq t_0+mn)\leq\phi_m^{-1}(n),\qquad n\in\mathbb{N},m>0.
    \end{equation*}
\end{thm}

\begin{proof}
  For $t\geq t_0,$ let $f=\frac{(\rho-t)^{+}}{m}\wedge 1$. One has $\Osc(f)\leq 1$ and
  \begin{equation*}
    \mu(f^p)^{2/p}\leq(\mu(f^2)^{p/2}\mu(\mathbf{1}_{\{f>0\}})^{(2-p)/2})^{2/p}\leq\mu(f^2)(1-a)^{(2-p)/p}.
  \end{equation*}
  By \eqref{eq:condition},
  \begin{equation*}
    \mathcal{E}(f,f)\leq m^{-2}\mu(\mathbf{1}_{\{t\leq\rho<t+m\}})=m^{-2}\mu(t\leq\rho<t+m).
  \end{equation*}
  By combining this with \eqref{eq:weak-poly} for $f$, we obtain
  \begin{equation*}
    \frac{1}{2-p}\mu(f^2)(1-(1-a)^{(2-p)/p})\leq m^{-2}\alpha(r)\mu(t\leq\rho<t+m)+r.
  \end{equation*}
  Letting $h(s):=\mu(\rho\geq s)$ for $s>0$, we arrive at
  \begin{equation*}
    p_0h(t+m)\leq\frac{1}{2-p}(1-(1-a)^{(2-p)/p})\mu(f^2)\leq m^{-2}\alpha(r)(h(t)-h(t+m))+r,
  \end{equation*}
  Hence,
  \begin{equation*}
    h(t+m)\leq\frac{m^{-2}\alpha(r)h(t)+r}{p_0+m^{-2}\alpha(r)},\qquad r>0.
  \end{equation*}
  By letting $r=\epsilon p_0 h(t)$ we obtain
  \begin{equation*}
    -\int_{h(t+m)}^{h(t)}ds=h(t+m)-h(t)\leq -\frac{(1-\epsilon)p_0h(t)}{p_0+m^{-2}\alpha(\epsilon p_0h(t))}.
  \end{equation*}
  Since $h(t)$ is decreasing in t and $(p_0+m^{-2}\alpha(\epsilon p_0s))/s$ is decreasing in s, we arrive at
  \begin{equation*}
    1\leq\frac{p_0+m^{-2}\alpha(\epsilon p_0h(t))}{(1-\epsilon)p_0 h(t)}\int_{t+m}^{t} dh(s)\leq\int_{h(t+m)}^{h(t)}\frac{p_0+m^{-2}\alpha(\epsilon p_0s)}{(1-\epsilon)p_0s}ds.
  \end{equation*}
  Thus,
  \begin{equation*}
    \phi_m(h(t_0+mn))=\int_{h(t_0+mn)}^{h(t_0)}\frac{p_0+m^{-2}\alpha(\epsilon p_0s)}{(1-\epsilon)p_0s}ds\geq n.
  \end{equation*}
  Therefore, $\mu(\rho\geq t_0+mn)=h(t_0+mn)\leq\phi_m^{-1}(n)$.
\end{proof}
At $p=1$, Theorem~\ref{thm:concentration} reduces to Theorem~4.2.1 of
\cite{Wangbook}. At $p=2$, the coefficient is interpreted by continuity as
\begin{equation*}
  p_0=\lim_{p\rightarrow 2}\frac{1-(1-a)^{(2-p)/p}}{2-p}=\frac{1}{2}\ln\frac{1}{1-a}.
\end{equation*}

\begin{cor}\label{cor:concentration}
  Under the assumptions of Theorem~\ref{thm:concentration}:
  \begin{enumerate}
    \item[(1)] If $\alpha(r)=cr^{-\delta}$ for some $c,\delta>0$ and all $r\in(0,1]$, then $\mu(\rho>N)\leq c'(\frac{N}{\log N})^{-2/\delta}$ for some $c'>0$ and all $N>1$.
    \item[(2)] If $\alpha(r)=c[\log r^{-1}]^{\delta}$ for some $c,\delta>0$ and all $r\in (0,1/2]$, then
    \begin{equation*}
      \mu(\rho>N)\leq c_1\exp[-c_2 N^{1/(1+\delta)}], N\geq 1
    \end{equation*}
    for some $c_1,c_2>0$.
  \end{enumerate}
\end{cor}

\begin{proof}
  It suffices to consider sufficiently large $N$.
  \begin{enumerate}
    \item[(1)] If $\alpha(r)=cr^{-\delta}$, then
    \begin{align*}
      \phi_m(t)&=\int_t^{1-a}\frac{p_0+m^{-2}c(\epsilon p_0s)^{-\delta}}{(1-\epsilon)p_0s}ds\\
      &\leq c_1\left(\frac{1}{m^2}\int_t^{1-a}s^{-(1+\delta)}ds+\log t^{-1}\right)\leq c_2(m^{-2}t^{-\delta}+\log t^{-1})
    \end{align*}
    for some $c_1,c_2>0$ and all $t\leq (1-a)/2$. Therefore,
    \begin{equation*}
      \phi_m^{-1}(n):=\inf\{t>0:\phi_m(t)\leq n \}\leq c_3(e^{-n}\vee m^{-2/\delta})
    \end{equation*}
    for some constant $c_3>0$ and large $n$. Thus the first assertion follows by taking $n$ the integer part of $\frac{2}{\delta}\log N$.
    \item[(2)] If $\alpha(r)=c[\log r^{-1}]^{\delta}$ for some $c>0$ and all
    $r\in(0,1/2]$, then the desired estimate follows from
    Theorem~\ref{thm:concentration} and
    \begin{equation*}
      \phi_1(t)\leq c_1\int_t^{1-a}s^{-1}[\log s^{-1}]^{\delta} ds \leq c_2(\log t^{-1})^{1+\delta}
    \end{equation*}
    for some $c_1,c_2>0$ and sufficiently small $t>0$.
  \end{enumerate}
\end{proof}
The following example shows that the estimates in
Corollary~\ref{cor:concentration} are sharp.

\begin{exa}
  Consider $L=\frac{d^2}{dx^2}+b(x)\frac{d}{dx}$ on $[0,\infty)$ with a
  reflecting boundary at $0$. Let $V(x):=\int_0^x b(s)\,ds$. Then $L$ is
  symmetric with respect to the probability measure
  $\mu(dx)=Z^{-1}e^{V(x)}\,dx$, where
  $Z=\int_0^\infty e^{V(x)}\,dx$, provided $Z<\infty$. Moreover,
  $\Gamma(f,f):=\frac12L(f^2)-fLf=(f')^2$.
  \begin{enumerate}
    \item[(1)] Let $V(x)=-q\log(1+x)$ for some $q>1$, so that
    $b(x)=-q/(1+x)$. Then Corollary~\ref{cor:concentration}(1) yields the
    correct leading-order decay of $\mu(\rho>N)$ as $N\to\infty$.
    \item[(2)] Let $V(x)=-(1+x)^{\delta}$ for some $\delta\in(0,1]$, so that
    $b(x)=-\delta(1+x)^{\delta-1}$. Then
    Corollary~\ref{cor:concentration}(2) yields the correct leading-order decay
    of $\mu(\rho>N)$.
  \end{enumerate}
  The proof follows from Proposition~4.2.3 of \cite{Wangbook}.
\end{exa}

Finally, let us consider the general setting without assuming \eqref{eq:condition}.
\begin{thm}
  Let $(\mathcal{E},\mathcal{D}(\mathcal{E}))$ be a conservative symmetric Dirichlet form on $L^2(\mu)$ for a probability measure $\mu$. For any $\rho\in\mathcal{D}(L_{\mathcal{E}})$ with $L_{\mathcal{E}}(\rho)\leq 1,$ \eqref{eq:weak-poly} implies
  \begin{equation*}
    \mu(\rho\geq t_0+N)\leq \frac{2-p}{1-(1-\mu(\rho<t_0))^{(2-p)/p}}\inf_{r>0}\left(r+\frac{2}{N^2}\alpha(r)\right),\quad t_0,N>0.
  \end{equation*}
\end{thm}
\begin{proof}
  Let $\rho_N:=(\rho-t_0)^{+}/N$, by lemma 1.2.4 in \cite{Wangbook} we have
  \begin{equation*}
    \mathcal{E}(\rho_N\wedge 1,\rho_N\wedge 1)\leq\mathcal{E}(\rho_N,\rho_N)\leq\frac{2}{N^2}.
  \end{equation*}
  Similar to the proof of Theorem~\ref{thm:concentration}, combining this with \eqref{eq:weak-poly} for $f:=\rho_N\wedge 1$ we obtain
  \begin{equation*}
    \frac{1}{2-p}\mu(f^2)\left(1-(1-\mu(\rho<t_0))^{(2-p)/p}\right)\leq\frac{2}{N^2}\alpha(r)+r.
  \end{equation*}
  This implies the desired assertion by noting that $\mu(\rho\geq t_0+N)\leq \mu(f^2)$.
  
  At $p=1$, this theorem reduces to Theorem~4.2.4 of \cite{Wangbook}:
  \begin{equation*}
    \mu(\rho>t_0+N)\leq\frac{1}{\mu(\rho<t_0)}\inf_{r>0}\left(r+\frac{2}{N^2}\alpha(r)\right).
  \end{equation*}
  where at $p=2$ it becomes
  \begin{equation*}
    \mu(\rho>t_0+N)\leq \frac{2}{\ln\frac{1}{\mu(\rho<t_0)}}\inf_{r>0}\left(r+\frac{2}{N^2}\alpha(r)\right).
  \end{equation*}
\end{proof}

\section{Appendix}
The following one-variable inequality forms the algebraic core of the proof.

\begin{lem}\label{lem:algebra}
Fix $p\in(1,2)$ and set
\[
  a_p=\frac{2(2-p)}p,\qquad
  \lambda_p=\frac{2(p^3-2p^2+2)}{p^2(p-1)},\qquad
  \eta_p=\frac{2(p-2)}{p-1}.
\]
Then, for every $x\ge0$,
\begin{equation}\label{eq:fundamental}
  x^2-1-a_px^p\log x
  \ge
  \lambda_p(x^p-1)+\eta_p(x-x^{p-1}),
\end{equation}
with equality if and only if $x=1$.
\end{lem}

\begin{proof}
Let
\[
  F_p(x)
  =x^2-1-a_px^p\log x
   -\lambda_p(x^p-1)-\eta_p(x-x^{p-1}).
\]
For $x>0$, direct differentiation gives
\[
  F_p'(x)=\frac{2x^{p-2}}{p-1}H_p(x),
\]
where
\[
\begin{split}
  H_p(x)={}&(p-1)(p-2)x\log x+(-p^2+3p-3)x\\
  &+(p-1)(p-2)
  +x^{2-p}\bigl((p-1)x+2-p\bigr).
\end{split}
\]
Moreover $H_p(1)=H_p'(1)=0$. Put $s=2-p\in(0,1)$. A second
differentiation yields
\[
  H_p''(x)=\frac{(p-2)(p-1)}xK_s(x),
  \qquad
  K_s(x)=1-(1+s)x^s+sx^{s-1}.
\]
Since
\[
  K_s(1)=0,\qquad
  K_s'(x)=sx^{s-2}\bigl((s-1)-(s+1)x\bigr)<0
\]
for $x>0$, we have $K_s>0$ on $(0,1)$ and $K_s<0$ on
$(1,\infty)$. Hence $H_p''<0$ on $(0,1)$ and $H_p''>0$ on
$(1,\infty)$. Together with $H_p'(1)=0$, this implies
$H_p'(x)>0$ for $x\ne1$. Thus $H_p$ is strictly increasing and,
because $H_p(1)=0$,
\[
  H_p(x)<0\quad(0<x<1),\qquad
  H_p(x)>0\quad(x>1).
\]
Therefore $F_p$ decreases on $(0,1)$ and increases on $(1,\infty)$.
Since $F_p(1)=0$, $x=1$ is its unique global minimum. Finally,
\[
  F_p(0)=\lambda_p-1
  =\frac{(p-2)^2(p+1)}{p^2(p-1)}\ge0,
\]
so \eqref{eq:fundamental} also holds at $x=0$.
\end{proof}

\section*{Acknowledgments}

We would like to thank Professor Feng-Yu Wang for useful conversations. This research is supported by the National Natural Science
Foundation of China (NNSFC), Grant No.~12571154.

\end{document}